\documentclass[11pt]{amsart}
\usepackage{amsmath,amsthm,amssymb, amscd, amsfonts}
\usepackage[all]{xy}
\usepackage{leftidx}
\usepackage[latin1]{inputenc}
\usepackage{enumerate}
\usepackage[dvipsnames]{xcolor}

\theoremstyle{plain}

\newtheorem{theorem}{Theorem}[section]
\newtheorem{corollary}[theorem]{Corollary}

\newtheorem{proposition}[theorem]{Proposition}
\newtheorem{lemma}[theorem]{Lemma}

\theoremstyle{definition}
\newtheorem{definition}[theorem]{Definition}
\newtheorem{claim}[theorem]{Claim}
\newtheorem{example}[theorem]{Example}
\newtheorem*{theorem1}{Theorem 5.5}

\theoremstyle{remark}
\newtheorem{remark}[theorem]{Remark}

\numberwithin{equation}{section}\theoremstyle{plain}

\newcommand{\I}{\mathcal{I}}

\renewcommand{\1}{\textbf{1}}

\newcommand{\A}{{\mathcal A}}
\newcommand{\B}{{\mathcal B}}
\newcommand{\C}{{\mathcal C}}
\newcommand{\D}{{\mathcal D}}

\newcommand{\Z}{{\mathcal Z}}
\newcommand{\Zz}{{\mathbb Z}}

\newcommand{\Ii}{\mathfrak{I}}

\newcommand{\E}{{\mathcal E}}

\newcommand{\Rep}{\operatorname{Rep}}
\newcommand{\cd}{\mathrm{cd}}

\newcommand\Aut{\operatorname{Aut}}
\newcommand\Irr{\operatorname{Irr}}
\newcommand\FPdim{\operatorname{FPdim}}

\newcommand\vect{\operatorname{Vec}}
\newcommand\svect{\operatorname{sVec}}
\newcommand\id{\operatorname{id}}

\newcommand\Hom{\operatorname{Hom}}

\begin{document}
\title{A class of prime fusion categories of dimension $2^N$}	

\author{Jingcheng Dong}
\address{College of Mathematics and Statistics, Nanjing University of Information Science and Technology, Nanjing 210044, China}
\email{jcdong@nuist.edu.cn}

\author{Sonia Natale}
\address{Facultad de Matem\' atica, Astronom\' \i a, F\' \i sica y Computaci\' on,
Universidad Nacional de C\' ordoba, CIEM -- CONICET, (5000) Ciudad
Universitaria, C\' ordoba, Argentina}
\email{natale@famaf.unc.edu.ar \newline \indent \emph{URL:}\/
http://www.famaf.unc.edu.ar/$\sim$natale}

\author{Hua Sun}
\address{Department of Mathematics, Yangzhou University, Yangzhou, Jiangsu 225002, China}
\email{d160028@yzu.edu.cn}



\date{October 21, 2019}

\begin{abstract} We study a class of strictly weakly integral fusion categories $\mathfrak{I}_{N, \zeta}$, where $N \geq 1$ is a natural number and $\zeta$ is a $2^N$th root of unity, that we call $N$-Ising fusion categories. An $N$-Ising fusion category has Frobenius-Perron dimension $2^{N+1}$ and is a graded extension of a pointed fusion category of rank 2 by the cyclic group of order $\mathbb Z_{2^N}$. We show that every braided $N$-Ising fusion category is prime and also that there exists a slightly degenerate $N$-Ising braided fusion category for all $N > 2$. We also prove a structure result for braided extensions of a rank 2 pointed fusion category in terms of braided $N$-Ising fusion categories.
\end{abstract}

\maketitle

\section{Introduction}\label{sec1}

Among the most basic examples of fusion categories, the \emph{pointed fusion categories} are those whose simple objects are invertible. A pointed fusion category is determined by its group of invertible objects $G$ and the cohomology class of a 3-cocycle $\omega$ on $G$, who is responsible for the associativity constraint. We denote by $\vect_{G}^\omega$ the pointed fusion category associated to the pair $(G, \omega)$.

\medbreak
Let $G$ be a finite group. A fusion category $\C$ is called a \emph{$G$-extension} of a fusion category $\D$ if it admits a faithful grading by the group $G$,
$$\C=\oplus_{g\in G}\C_g,$$
such that $\C_g \otimes \C_h \subseteq \C_{gh}$, for all $g, h \in G$, and the trivial homogeneous component is equivalent to $\D$ \cite{gelaki2008nilpotent}.
Thus, a fusion category $\C$ is pointed if and  only if $\C$ is a $G$-extension of the fusion category $\vect$ of finite dimensional vector spaces, for some finite group $G$.

\medbreak
An \emph{Ising category} is a fusion category of Frobenius-Perron dimension $4$ which is not pointed. Ising categories appear in Conformal Field Theory related to 2-dimensional Ising models.

Every Ising fusion category is a $\Zz_2$-extension of the rank 2 pointed fusion category $\vect_{\Zz_{2}}$ and it belongs to the class of fusion categories classified by Tambara and Yamagami in  \cite{TY}; in particular there exist exactly 2 Ising fusion categories, up to equivalence, and they are a 3-cocycle twist of each other.

By the main result of \cite{siehler-braided}, every Ising fusion category admits exactly 4 non-equivalent braidings. In particular all such braidings are non-degenerate. Several properties of Ising fusion categories are studied in \cite[Appendix B]{drinfeld2010braided}.  See Subsection \ref{sec2.3}.

\medbreak
In this paper we study a family of examples of fusion categories that are obtained from Ising fusion categories and share some features with them. We call them \emph{$N$-Ising fusion categories}. They are special instances of the cyclic extensions of adjoint categories of ADE type classified in \cite{EM} and are defined as follows:
Let $\Ii$ be the semisimplification of the representation category of $U_{-q}(\mathfrak{sl}_2)$, with $q = \exp(i\pi/4)$. Then $\Ii$ is an Ising fusion category. Let $Z$ be the non-invertible simple object of $\Ii$. Then an $N$-Ising category is defined as a 3-cocycle twist of the fusion subcategory of $\Ii \boxtimes \vect_{\Zz_{2^N}}$ generated by the simple object $Z \boxtimes 1$; c.f. Section \ref{nising}. (The definition of a 3-cocycle twist of a group-graded fusion category is recalled in Subsection \ref{sec2.2}.)

\medbreak
A $1$-Ising fusion category is thus an Ising fusion category. For every $N\geq 1$, an $N$-Ising fusion category has Frobenius-Perron dimension $2^{N+1}$ and is a graded extension of a pointed fusion category of rank 2 by the cyclic group of order $\mathbb Z_{2^N}$. In addition every $N$-Ising fusion category is strictly weakly integral. Its group of invertible objects is isomorphic to $\Zz_2 \times \Zz_{2^{N-1}}$ and it has $2^{N-1}$ simple objects of Frobenius-Perron dimension $\sqrt 2$, none of which is self-dual except in the case $N = 1$.

\medbreak
As graded extensions of $\vect_{\Zz_2}$, $N$-Ising fusion categories are parameterized by the integer $N$ and a $2^N$th root of unity $\zeta$. The corresponding category is denoted $\Ii_{N, \zeta}$. We use the notation $\Ii_N$ to indicate the category $\Ii_{N, 1}$.

Every $N$-Ising fusion category $\Ii_{N, \pm 1}$ admits the structure of a braided fusion category.  We show that a braided $N$-Ising fusion category is always prime (Corollary \ref{in-prime}), that is, it does not contain any non-trivial non-degenerate fusion subcategory.
We also show that with respect to any possible braiding, an $N$-Ising fusion category is non-degenerate if and only if $N = 1$. In addition, a slightly degenerate braided $N$-Ising category exists if $N > 2$. See Subsection \ref{braided-nisings}.

Observe that, as shown \cite{EM}, when $N \geq 2$ there is another family of non-pointed $\Zz_{2^N}$-extensions  of $\vect_{\Zz_2}$, not equivalent to any $N$-Ising fusion category. However, the fusion categories in this family do not admit any braiding (Theorem \ref{braided-cyclic}).

\medbreak
Our main result for braided extensions of a rank 2 pointed fusion category  is the following theorem:

\begin{theorem1}Let $\C$ be a non-pointed braided fusion category and suppose that $\C$ is an extension of a rank 2 pointed fusion category. Then $\C$ is equivalent as a fusion category to $\I_{N} \boxtimes \B$, for some $N\geq 1$, where $\I_N$ is a braided $N$-Ising fusion category, and $\B$ is a pointed braided fusion category.
Furthermore, the categories $\I_{N}$ and $\B$ projectively centralize each other in $\C$.
\end{theorem1}

The notion of projective centralizer of a fusion subcategory, introduced in \cite{drinfeld2010braided}, is recalled in Subsection \ref{sec2.2}.

\medbreak
Theorem \ref{cls-extensions-vec2} is proved in Section \ref{structure}. Its proof relies on the classification results of \cite{EM}.
We point out that Theorem \ref{cls-extensions-vec2} applies in particular when $\C$ is a slightly degenerate braided fusion category with generalized Tambara-Yamagami fusion rules, that is, when $\C$ is slightly degenerate, not pointed, and the tensor product of two non-invertible simple objects decomposes as a sum of invertible objects.

\medbreak The paper is organized as follows. In Section \ref{sec2} we discuss some preliminary notions and results on fusion categories that will be relevant in the rest of the paper. Section \ref{sec42} contains some basic results on the structure of a general group extension of a rank 2 pointed fusion category and on braided such extensions that will be needed in the sequel. In Section \ref{nising} we introduce $N$-Ising categories and study their main properties.
In Section \ref{structure} we give a proof of our main result on braided extensions of a rank 2 pointed fusion category.

\section{Preliminaries}\label{sec2}
We shall work over an algebraically closed field $k$ of characteristic zero. A fusion category over $k$ is a $k$-linear semisimple rigid tensor category with finitely many isomorphism classes of simple objects, finite-dimensional vector spaces of morphisms and the unit object $\1$ is simple. We refer the reader to \cite{etingof2005fusion}, \cite{drinfeld2010braided} for the main notions on fusion categories and braided fusion categories used throughout.

\medbreak
Let $\C$ be a fusion category. The tensor product in $\C$ induces a ring structure in the Grothendieck ring $K(\C)$ of $\C$. By \cite[Section 8]{etingof2005fusion}, there is a unique ring homomorphism $\FPdim:K(\C)\to\mathbb{R}$ such that $\FPdim(X)\geq 1$ for all nonzero $X\in \C$. The number $\FPdim(X)$ is called the Frobenius-Perron dimension of $X$. The Frobenius-Perron dimension of $\C$ is defined by
$$\FPdim(\C)=\sum_{X\in\Irr(\C)}\FPdim(X)^2,$$
where $\Irr(\C)$ is the set of isomorphism classes of simple objects in $\C$.

\medbreak
A simple object $X\in \C$ is called invertible if $X\otimes X^*\cong \1$, where $X^*$ is the dual of $X$. Thus $X$ is invertible if and only if $\FPdim(X)=1$. A fusion category $\C$ is called pointed if every simple object of $\C$ is invertible. Pointed fusion categories whose group of invertible objects is isomorphic to $G$ are classified by the orbits of the action of the group $\textrm{Out}(G)$ in $H^3(G,k^\times)$. The pointed fusion category corresponding to the class of a 3-cocycle $\omega$ will be denoted by $\vect_{G}^{\omega}$.

The largest pointed subcategory of $\C$, denoted $\C_{pt}$, is the fusion subcategory generated by all invertible simple objects.
The set $G = G(\C)$ of isomorphism classes of invertible objects of $\C$ is a finite group with multiplication given by tensor product. The inverse of $X\in G$ is its dual $X^*$.
The group $G$ acts on the set $\Irr(\C)$ by left tensor product multiplication. Let $G[X]$ be the stabilizer of $X\in \Irr(\C)$ under this action. Then we have a decomposition
\begin{equation}\label{decom1}
\begin{split}
X\otimes X^*=\bigoplus_{g\in G[X]}g\oplus\sum_{Y\in \Irr(\C)-G[X]}  \dim\Hom(Y,X\otimes X^*) \; Y.
\end{split}
\end{equation}

\subsection{Group extensions of fusion categories}\label{sec2.2}
Let $G$ be a finite group. A fusion category $\C$ is graded by $G$  if $\C$ has a direct sum decomposition into full abelian subcategories $\C=\oplus_{g\in G}\C_g$ such that  $\C_g\otimes\C_h\subseteq\C_{gh}$, for all $g,h\in G$. If $\C_g\neq 0$, for all $g\in G$, then the grading is called faithful. When the grading is faithful, $\C$ is called a $G$-extension of the trivial component $\C_e$.

\medbreak
If $\C=\oplus_{g\in G}\C_g$ is a faithful grading of $\C$, then \cite[Proposition 8.20]{etingof2005fusion} shows that
\begin{equation*}
\FPdim(\C)=|G| \FPdim(\C_e), \quad \FPdim(\C_g)=\FPdim(\C_h),\quad \forall g,h\in G.
\end{equation*}

It follows from the results of \cite{gelaki2008nilpotent} that every fusion category $\C$ has a canonical faithful grading $\C=\oplus_{g\in U(\C)}\C_g$ with trivial component $\C_e=\C_{ad}$, where $\C_{ad}$ is the adjoint subcategory of $\C$ generated by simple the constituents of $X\otimes X^*$, for all $X\in \Irr(\C)$. This grading is called the universal grading of $\C$, and $U(\C)$ is called the \emph{universal grading group} of $\C$. Any other faithful grading of $\C=\oplus_{g\in G}\C_g$ is determined by a surjective group homomorphism $\pi:U(\C)\to G$. Hence the trivial component $\C_e$ contains $\C_{ad}$.

\medbreak
Let $G$ be a finite group and let $\C$ be a $G$-extension of a fusion category $\D \cong \C_e$. Let also $\omega \in Z^3(G, k^{\times})$ be a 3-cocycle. We shall denote by $\C^\omega$ the fusion category obtained from $\C$ by twisting the associator with $\omega$. For $\omega_1, \omega_2 \in Z^3(G, k^{\times})$, the categories $\C^{\omega_1}$ and $\C^{\omega_2}$ are equivalent as $G$-extensions of $\D$ if and only if the classes of $\omega_1$ and $\omega_2$ coincide in $H^3(G, k^{\times})$. See \cite{ENO-homotopy}.

\subsection{Braided fusion categories}\label{sec2.2}
A braided fusion category $\C$ is a fusion category admitting a braiding $c$, that is, a family of natural isomorphisms: $c_{X,Y}$:$X\otimes Y\rightarrow Y\otimes X$, $X,Y\in\C$, obeying the hexagon axioms.

\medbreak
Let $\C$ be a braided fusion category. Then $X,Y\in \C$ are said to centralize each other if $c_{Y,X}c_{X,Y}=\id_{X\otimes Y}$. The centralizer $\D'$ of a fusion subcategory $\D\subseteq \C$ is the full subcategory of objects which centralize every object of $\D$, that is
$$\D^{\prime}=\{X\in \C \mid c_{Y,X}c_{X,Y}=\id_{X\otimes Y},\forall\ Y\in \D\}.$$

The M\"{u}ger center $\Z_2(\C)$ of a braided fusion category $\C$ is the centralizer $\C'$ of $\C$ itself. A braided fusion category $\C$ is called non-degenerate if $\Z_2(\C)$ is equivalent to the trivial category $\vect$. A braided fusion category $\C$ is called slightly degenerate if $\Z_2(\C)$ is equivalent to the category $\svect$ of super-vector spaces.

\medbreak Two full subcategories $\D$ and $\tilde \D$ of $\C$ are said to \emph{projectively centralize each other} if for all simple objects $X \in \D$ and $Y\in \tilde D$, the squared-braiding $c_{Y, X}c_{X, Y}$ is a scalar multiple of the identity $\id_{X\otimes Y}$. See \cite[Subsection 3.3]{drinfeld2010braided}.

Suppose that  $\D$ and $\tilde \D$ are fusion subcategories of $\C$ that projectively centralize each other. Then \cite[Proposition 3.32]{drinfeld2010braided} shows that there exist finite groups $G$ and $\tilde G$ endowed with a non-degenerate pairing $b: G \times \tilde G \to k^\times$ and faithful gradings $\D = \bigoplus_{g \in G}\D_g$, $\tilde \D = \bigoplus_{g \in \tilde G}{\tilde \D}_g$, such that $\D_0 = \D \cap \tilde\D'$, $\tilde\D_0 = \D' \cap \tilde\D$, and for all homogeneous simple objects $X \in \D_g$, ${\tilde \D}_h$, $g\in G$, $h \in \tilde G$, the squared-braiding $c_{Y, X}c_{X, Y}$ is given by
$$c_{Y, X}c_{X, Y} = b(g, h) \, \id_{X \otimes Y}.$$

\medbreak
A braided fusion category $\C$ is called symmetric if $\Z_2(\C)=\C$. Hence the M\"{u}ger center of a braided fusion category is a symmetric fusion category.

A symmetric fusion category $\C$ is called Tannakian if it is equivalent to the category $\Rep(G)$ of finite-dimensional representations of a finite group $G$, as braided fusion categories.

\medbreak
Let $\C$ be a symmetric fusion category. Deligne proved that there exist a finite group $G$ and a central element $u$ of order $2$, such that $\C$ is equivalent to the category $\Rep(G,u)$ of representations of $G$ on finite-dimensional super vector spaces, where $u$ acts as the parity operator \cite{deligne1990categories}.

The symmetric category $\C$ is either Tannakian or a $\mathbb{Z}_2$-extension of a Tannakian subcategory.
Therefore,  if $\FPdim(\C)$ is odd, then $\C$ is Tannakian. Moreover if $\FPdim(\C)$ is bigger than $2$ then $\C$ necessarily contains a Tannakian subcategory. Also, a non-Tannakian symmetric fusion category of Frobenius-Perron dimension $2$ is equivalent to the category $\svect$, see \cite[Subsection 2.12]{drinfeld2010braided}.

\begin{proposition}\label{Adjo_Point}
Let $\C$ be a braided fusion category. Then $\C_{ad}\subseteq (\C_{pt})'$.
\end{proposition}
\begin{proof}
Suppose first that $\C$ is non-degenerate. Then $\C_{ad}= (\C_{pt})'$ by \cite[Corollary 3.27]{drinfeld2010braided}.

Let now $\C$ be an arbitrary braided fusion category. The braiding of $\C$ induces a canonical embedding of braided fusion categories $\C\hookrightarrow \Z(\C)$. Hence, we may identify $\C$ with a fusion subcategory of $\Z(\C)$. We therefore have $\C_{pt}\subseteq \Z(\C)_{pt}$ and $\C_{ad}\subseteq \Z(\C)_{ad}$. This implies that $(\C_{pt})'\supseteq (\Z(\C)_{pt})'=\Z(\C)_{ad}\supseteq\C_{ad}$. The  equality holds true because $\Z(\C)$ is non-degenerate, also by \cite[Corollary 3.27]{drinfeld2010braided}. This completes the proof.
\end{proof}

The following theorem is due to Drinfeld et al. In the case when $\C$ is modular, it is due to M\"{u}ger \cite[Theorem 4.2]{muger2003structure}.

\begin{theorem}{\cite[Theorem 3.13]{drinfeld2010braided}}\label{MugerThm}
Let $\C$ be a braided fusion category and $\D$ be a non-degenerate subcategory of $\C$. Then $\C$ is braided equivalent to $\D\boxtimes \D'$, where $\D'$ is the centralizer of $\D$ in $\C$.
\end{theorem}

For a pair of  fusion subcategories $\A,\B$ of $\D$, we use $\A\vee \B$ to denote the smallest fusion subcategory of $\C$ containing $\A$ and $\B$. The following result will be  used frequently.
\begin{lemma}\cite[Corollary 3.11]{drinfeld2010braided}\label{double_centralzer}
Let $\C$ be a braided fusion category. If $\D$ is any fusion subcategory of $\C$ then $\D''=\D\vee\mathcal{Z}_2(\C)$.
\end{lemma}

\subsection{Pointed braided fusion categories}\label{pointed-bd}

We recall in this subsection some facts related to the classification of braided pointed fusion categories. We refer the reader to \cite{JS}, \cite{quinn}, \cite{drinfeld2010braided} for a detailed exposition.

\medbreak
Let $G$ be a finite abelian group.
An \emph{abelian 3-cocycle} on $G$ with values in $k^\times$ is a pair $(\omega, c)$, where $\omega: G \times G \times G \to k^\times$ is a normalized 3-cocycle and $\sigma: G \times G \to k^\times$ is a 2-cochain such that
\begin{equation*}
\omega(a, b, c) \, \omega(b, c, a) \, \sigma(a, bc) =
\omega(b, a, c) \, \sigma(a, b) \, \sigma(a, c),
\end{equation*}
for all $a, b, c \in G$.
Abelian 3-cocycles form an abelian group $Z^3_{ab}(G, k^\times)$. Let $B^3_{ab}(G, k^\times) \subseteq Z^3_{ab}(G, k^\times)$ is the subgroup of abelian coboundaries, that is, abelian 3-cocycles of the form $(du, u (u_{21})^{-1})$ where $u: G \times G \to k^\times$ is a normalized 2-cochain and $du(a, b, c) = u(b,c) \, u(ab,c)^{-1} \, u(a, bc) \, u(a, b)^{-1}$, $a, b, c \in G$.

The quotient $H^3_{ab}(G, k^\times) = Z^3_{ab}(G, k^\times)/B^3_{ab}(G, k^\times)$ is called the \emph{abelian cohomology group} of $G$ with coefficients in $k^\times$.
Every braiding of a pointed fusion category with group $G$ of invertible objects corresponds to an element of the group $H^3_{ab}(G, k^\times)$. In particular, given a normalized 3-cocycle $\omega$ and a 2-cochain $\sigma$ on $G$, we have that the rule $$\sigma_{a, b}\id_{ab}: a\otimes b \to b\otimes a, \qquad a, b \in G,$$ defines a braiding in the fusion category $\vect_G^\omega$ if and only if $(\omega, \sigma) \in Z^3(G, k^\times)$.

\medbreak A \emph{quadratic form} on $G$ with values in $k^\times$ is a map $q: G \to k^\times$ satisfying $q(g) = q(g^{-1})$ for all $g \in G$ and such that $b(a, b) = q(ab)q(a)^{-1}q(b)^{-1}$ defines a symmetric bicharacter $G \times G \to k^\times$. If $q$ is a quadratic form on $G$, then the pair $(G, q)$ is called a \emph{pre-metric group}.

\medbreak
To every abelian 3-cocycle $(\omega, \sigma)$ on $G$ one can associate a quadratic form on $G$ defined by
\begin{equation}\label{quadratic}
q(g) = \sigma(g, g), \quad g \in G.
\end{equation}
A result of Eilenberg and Mac Lane states that this correspondence defines a group isomorphism between the abelian cohomology group $H^3_{ab}(G, k^\times)$ and the abelian group of quadratic forms on $G$.

\medbreak
Moreover, the functor that associates to every pointed fusion category $\C$ the pre-metric group $(G, q)$, where $G$ is the group of invertible objects of $\C$ and $q$ is the quadratic form \eqref{quadratic}, where $\sigma$ is the braiding of $\C$, defines an equivalence  between the category of pointed fusion categories and braided functors up to braided isomorphism and the category of pre-metric groups.

Thus, two braided fusion categories $\C(G, q)$ and $\C(G, q')$ associated to the quadratic forms $q$ and $q'$ on $G$ are equivalent if and only if there exists an automorphism $\varphi$ of $G$ such that $q'(\varphi(g)) = q(g)$, for all $g \in G$.

\medbreak The squared-braiding of the braided fusion category $\C(G, q)$ associated to a quadratic form $q$ is given by the symmetric bilinear form  $\beta: G\times G \to k^\times$ associated to $q$, defined in the form $$\beta(a, b) = q(ab)q(a)^{-1}q(b)^{-1}, \qquad a, b \in G.$$

\medbreak
Let $M$ be a natural number and let $G = \Zz_M$ be the cyclic group of order $M$.
Let also $\zeta \in k^\times$ be an $M$th root of 1. Then $\zeta$ determines a 3-cocycle $\omega_\zeta$ on $\Zz_M$ where, for all $0 \leq i, j, \ell \leq M-1$,
\begin{equation}\label{omega-zm}
\omega_\zeta(i, j, \ell) = \begin{cases}1, \quad \text{if } j+\ell < M, \\ \zeta^{i}, \quad \text{if } j+\ell \geq M.\end{cases}
\end{equation}
The assignment $\zeta \mapsto \omega_\zeta$ gives rise to a group isomorphism between the group $\mathbb G_M$ of $M$th roots of 1 in $k$ and the group $H^3(\Zz_M, k^\times)$.  In particular  $H^3(\Zz_M, k^\times) \cong \Zz_M$.

We shall denote by $\vect_{\Zz_M}^{\zeta}$ the pointed fusion category corresponding to the 3-cocycle $\omega_\zeta$. Thus  $\vect_{\Zz_M}^{1} = \vect_{\Zz_M}$ and, if $M$ is even, $\vect_{\Zz_M}^{-1}$ is the pointed fusion category corresponding to the 3-cocycle $\omega_{-1}$ associated to $\zeta = -1 \in \mathbb G_M$.

\medbreak
Let $\xi \in k^\times$ such that $\xi^{M^2} = 1 = \xi^{2M}$. Then the pair $(\omega_{\xi^M}, \sigma_{\xi})$ is an abelian 3-cocycle on $G$ where, for all $0 \leq i, j, \ell \leq M-1$,
\begin{equation}\label{braiding-xi}
\sigma_{\xi}(i, j) = \xi^{ij}.
\end{equation}
Furthermore, this gives rise to a group isomorphism between $H^3_{ab}(\Zz_M, k^\times)$ and the group $\mathbb G_d$ of $d$th roots of $1$ in $k^\times$, where $d = \text{gcd}(M^2, 2M)$. See \cite[pp. 49]{JS}, \cite[Subsection 2.5.2]{quinn}.

Thus $\vect_{\Zz_M}^{\xi^M}$ is a braided fusion category whose square braiding is given by $\beta_\xi(i, j)\, \id_{i+j}: i+j \to i+j$, where $\beta_\xi: \Zz_M\times \Zz_M \to k^\times$ is the bilinear form defined as
$$\beta_\xi (i, j) = \xi^{2ij}, \quad 0\leq i, j < M.$$

The quadratic form $q:\Zz_M \to k^\times$ and the corresponding symmetric bilinear form on $\Zz_M$ associated to the braiding \eqref{braiding-xi} are given, respectively, by the formulas
\begin{equation}\label{sq-bd}q(j) = \xi^{j^2}, \qquad \beta(i,j) = \xi^{2ij},
\end{equation}
for all $0\leq i, j \leq M-1$.

\medbreak
Note that the condition $\xi^{2M} = 1$ forces $\xi^M = \pm 1$. In particular, for a fixed value of $\zeta = \pm 1$, there are exactly $M$ choices for $\xi$. Thus we obtain:

\begin{lemma}\label{braiding-zm} The pointed fusion category $\vect_{\Zz_M}^{\zeta}$ admits a braiding if and  only if $\zeta = \pm 1$. In addition we have:
	
(1)\, If $M$ is odd, $\vect_{\Zz_M}^{\zeta}$ does not admit any braiding unless $\zeta = 1$, and in this case, it admits exactly $M$ braidings up to equivalence.	

(2)\, If $M$ is even, then each of the categories $\vect_{\Zz_M}$ and $\vect_{\Zz_M}^{-1}$ admits exactly $M$ braidings, up to equivalence.
\end{lemma}

\begin{example}\label{braided-z2n} Let $N\geq 1$ and let $\xi\in k^\times$ be a $2^{N+1}$th root of 1. It follows from formulas \eqref{sq-bd} that the braided fusion category associated to $\xi$ is non-degenerate if and only if $\xi$ is primitive. If this is the case, then the underlying fusion category is $\vect_{\Zz_{2^N}}^{-1}$.
	
\medbreak	
Let $\xi\in k^\times$ be a primitive 8th root of 1. Let $\C = \C(\Zz_4, \xi)$ be the corresponding (non-degenerate) braided fusion category.  From formulas \eqref{sq-bd}, we get that $q(2) = \xi^4 = -1$, hence in this case the subcategory $\langle 2\rangle \subseteq \C$ is equivalent to $\svect$.
\end{example}

\subsection{Ising categories}\label{sec2.3}
An Ising category is a fusion category of Frobenius-Perron dimension $4$ which is not pointed.
Let $\I$ be an Ising fusion category. Then, up to isomorphism, $\I$ has a unique nontrivial invertible object  $\delta$  and  a unique non-invertible simple object $Z$. Thus $\FPdim Z = \sqrt 2$ and the fusion rules of $\I$ are determined by the relation
\begin{equation}\label{fr-ising} Z^{\otimes 2} \cong \1 \oplus \delta.
\end{equation}

In view of the results of \cite{TY}, there exist exactly 2 non-equivalent Ising fusion categories. The universal grading group of $\I$ is isomorphic to $\Zz_2$. The explicit formulas for the associators of Ising categories in \cite{TY} imply that if $\I^+$ and $\I^-$ are two non-equivalent Ising categories then, up to an equivalence of fusion categories, any of them is obtained from the other by twisting the associator by the 3-cocycle $\omega_{-1}$ on $\Zz_2$ determined by the relation $\omega_{-1}(1, 1, 1) = -1$.

\medbreak
Every Ising fusion category admits a braiding and all possible braidings are classified by the main result of \cite{siehler-braided} (see also \cite{drinfeld2010braided}); in particular all such braidings are non-degenerate. The category $\I_{pt}$ is equivalent to the category $\svect$ of super-vector spaces as a braided fusion category.

\subsection{Equivariantizations and de-equivariantizations}\label{sec23}
Let $\C$ be a fusion category with an action by tensor autoequivalences $\rho: \underline{G} \to \Aut_\otimes(\C)$ of a finite group $G$. The equivariantization $\C^G$ of $\C$ under the action of $G$ is defined as the category of $G$-equivariant objects and morphisms in $\C$. An object of $\C^G$ is a pair $(X,(u_g)_{g\in G})$, where $X$ is an object of $\C$, $u_g : \rho^g(X)\to X$, $g\in G$, is an isomorphism such that
$$u_{gh}\circ \rho^2_{g,h} = u_g \circ \rho^g(u_h),$$
for all $g, h \in G$, where $\rho^2_{g,h}: \rho^g(\rho^h(X))\to \rho^{gh}(X)$ is the monoidal structure of the action $\rho$. The tensor product of equivariant objects is defined by means of the monoidal structure of the action.

\medbreak
Let $\C$ be a fusion category and let
$\E = \Rep(G)\subseteq \mathcal{Z}(\C)$ be a Tannakian subcategory of the Drinfeld center $\Z(\C)$ of $\C$ that embeds into $\C$ via the forgetful functor $\mathcal{Z}(\C)\to \C$. Then the algebra $A=k^G$ of $k$-valued functions on $G$ is a commutative algebra in $\mathcal{Z}(\C)$. The de-equivariantization $\C_G$ of $\C$ by $\E$ is the fusion category defined as the category of left $A$-modules in $\C$.  See \cite{drinfeld2010braided} for details on equivariantizations and de-equivariantizations.

Equivariantizations and de-equivariantizations are inverse to each other:
$(\C_G)^G\cong\C\cong(\C^G)_G$. As for their Frobenius-Perron dimensions, we have
\begin{equation*}
\begin{split}
\FPdim(\C)=|G|\FPdim(\C_G), \quad \FPdim(\C^G)=|G| \FPdim(\C).
\end{split}
\end{equation*}

An object of a fusion category $\C$  is called \emph{trivial} if it is isomorphic to $\1^{\oplus n}$ for some natural number $n$.

Given a Tannakian subcategory $\Rep(G)$ of a braided fusion category $\C$, we have an exact sequence of fusion categories (see \cite[Section 1]{bruguieres2011exact}):
\begin{equation*}
\begin{split}
\Rep(G)\hookrightarrow \C\xrightarrow{F}\C_G,
\end{split}
\end{equation*}
where $\C_G$ is the de-equivariantization of $\C$ by $\Rep(G)$, $F$ is the forgetful functor. Hence $\Rep(G)$ is the kernel of $F$, that is, the subcategory of $\C$ whose objects have trivial image under $F$.

\section{Extensions of a rank 2  pointed fusion category}\label{sec42}
\subsection{General Results}

Recall that a \emph{generalized Tambara-Yamagami fusion category} is a fusion category $\C$ which is not pointed and the tensor product of two non-invertible simple objects of $\C$ is a sum of invertible objects. See \cite{liptrap2010generalized}.

\begin{theorem}\label{General_TY_Cat}
Let $\C$ be a $G$-extension of a pointed fusion category $\vect_{\mathbb{Z}_2}^{\omega}$. Then

(1)\, If $\omega = -1$ then $\C$ is pointed.

(2)\, If $\omega = 1$ then $\C$ is either pointed or a  generalized Tambara-Yamagami fusion category. If the last possibility holds, then:

\mbox{\hspace{0.6cm}}{(i)} Up to isomorphism, $\C$ has $2n$ invertible objects and $n$ simple objects of Frobenius-Perron dimension $\sqrt 2$, for some $n \geq 1$.

\mbox{\hspace{0.6cm}}{(ii)} $\C_{ad} \cong \vect_{\mathbb{Z}_2}$ as fusion categories, and $U(\C) = G$ is of order $2n$.
\end{theorem}
\begin{proof}
Let $\C=\oplus_{g\in G}\C_g$ be the corresponding faithful grading such that $\C_e=\vect_{\mathbb{Z}_2}^{\omega}$. Since this grading is faithful, every component $\C_g$ has Frobenius-Perron dimension $2$. Since $\C$ is weakly integral, the Frobenius-Perron dimension of every simple object is a square root of some integer \cite[Proposition 8.27]{etingof2005fusion}. This implies that every component $\C_g$ either contains $2$ non-isomorphic invertible objects, or it contains a unique $\sqrt 2$-dimensional simple object.
If $\C$ is not pointed, then the trivial component $\C_e$ is pointed and there exists a component $\C_g$ containing a unique $\sqrt2$-dimensional simple object. It follows from \cite[Lemma 2.6]{jordan2009classification} that $\omega$ is trivial. Then (1) holds.

\medbreak
Suppose that $\C$ is not pointed. By \cite[Theorem 3.10]{gelaki2008nilpotent}, $\C$ is endowed with a faithful $\mathbb{Z}_2$-grading $\C=\oplus_{h\in \mathbb{Z}_2}\C^h$, where the trivial component $\C^0$ is $\C_{pt}$ and $\C^{1}$ contains all $\sqrt2$-dimensional simple objects. Let $X,Y$ be non-invertible simple objects of $\C$. Then $X, Y \in \C^{1}$ and hence $X\otimes Y\in \C^0$, which implies that $X\otimes Y$ is a direct sum of invertible objects. Hence $\C$ is a generalized Tambara-Yamagami fusion category.

\medbreak
Assume that the number of non-isomorphic $\sqrt 2$-dimensional simple objects  is $n \geq 1$. Then $2n = \FPdim(\C^{1}) = \FPdim(\C^0)$. Hence $|G| = 2n$  and we get part (i).

\medbreak
Since $\C_{ad}\subseteq \C_e\cong \vect_{\mathbb{Z}_2}$, we know $\C_{ad}=\vect$ or $\vect_{\mathbb{Z}_2}$. Since $\C$ is not pointed, then $\C_{ad}$ can not be $\vect$. Therefore $\C_{ad}=\C_e$ and $G=U(\C)$. In particular the order of $U(\C)$ is $2n$. This proves part (ii).
\end{proof}

For a fusion category $\C$, let $\cd(\C)$ denote the set of Frobenius-Perron dimensions of simple objects of $\C$.

\begin{corollary}\label{necessary_sufficent}
Let $\C$ be a non-pointed fusion category. Then $\C$ is an extension of a rank $2$ pointed fusion category if and only if $\cd(\C)=\{1,\sqrt{2}\}$.
\end{corollary}
\begin{proof}
In view of Theorem \ref{General_TY_Cat}, it will be enough to show that the condition $\cd(\C)=\{1,\sqrt{2}\}$ implies that $\C$ is an extension of a rank $2$ pointed fusion category. So assume that $\cd(\C)=\{1,\sqrt{2}\}$.

As in the proof of Theorem \ref{General_TY_Cat} we get that $\C$ is a generalized Tambara-Yamagami fusion category. Then, by \cite[Proposition 5.2]{natale2013faithful}, the adjoint subcategory $\C_{ad}$ coincides with the fusion subcategory generated by $G[X]$, for any $\sqrt{2}$-dimensional simple object $X$. Hence $\FPdim(\C_{ad})=2$ and $\C$ is an extension of a rank $2$ pointed fusion category.
\end{proof}

\begin{corollary}\label{transitive-action}
Let $\C$ be a $G$-extension of $\vect_{\mathbb{Z}_2}$. Assume that $\C$ is not pointed. Then

(1)\,  The action  of the group $G(\C)$ by left (or right) tensor multiplication on the set of non-invertible simple objects of $\C$ is transitive.

(2)\, The group $\mathbb{Z}_2$ is a normal subgroup of $G(\C)$.
\end{corollary}

\begin{proof}
Since $\C$ is not pointed, Theorem \ref{General_TY_Cat} implies that $\C$ is a  generalized Tambara-Yamagami fusion category. The corollary then follows from \cite[Lemma 5.1]{natale2013faithful}.
\end{proof}

\subsection{Braided extensions of $\vect_{\mathbb{Z}_2}$}

Throughout this subsection $\C$ will be an extension of  $\vect_{\mathbb{Z}_2}$ unless explicitly stated. In addition, we assume that $\C$ is braided and not pointed.

\begin{lemma}\label{cad}
The adjoint subcategory $\C_{ad}$ is equivalent to $\svect$ as braided fusion categories.
\end{lemma}
\begin{proof}
By Theorem \ref{General_TY_Cat}, we know that $\C_{ad}\cong \vect_{\mathbb{Z}_2}$.
By \cite[Lemma 2.5]{2016GonNatale}, $\C_{ad} = \C_{ad}\cap C_{pt}$ is symmetric.  Suppose on the contrary that $\C_{ad}$ is Tannakian. Then $\C_{ad}\cong \Rep(\mathbb{Z}_2)$ as braided fusion categories and $\C$ is a $\mathbb{Z}_2$-equivariantization of a fusion category $\C_{\mathbb{Z}_2}$.

The forgetful functor $F:\C\rightarrow \C_{\mathbb{Z}_2}$ is a tensor functor and the image of every object in $\C_{ad}$ under $F$ is a trivial object of $\C_{\mathbb{Z}_2}$. Let $\delta$ be the unique non-trivial simple object of $\C_{ad}$. If $X$ is a non-invertible simple object of $\C$ then $X\otimes X^{*}\cong \1\oplus\delta$. Hence $F(X\otimes X^{*})\cong F(X)\otimes F(X)^{*}\cong \1\oplus \1$, which implies that $F(X)$ is not simple. On the other hand, if  $F(X)$ is not simple then the decomposition of $F(X)\otimes F(X)^{*}$ should contain at least four simple summands. This contradiction shows that the assumption is false, and therefore $\C_{ad}\cong \svect$, as claimed.
\end{proof}

Recall that if $\D$ is a fusion category with commutative Grothendieck ring and $\A$ is a fusion subcategory of $\D$, the  \emph{commutator} of $\A$ in $\D$, denoted by $\A^{co}$,  is the fusion subcategory of $\D$ generated by all simple objects $X$ of $\D$ such that $X\otimes X^*$ is contained in $\A$ \cite{gelaki2008nilpotent}.

\begin{lemma}\label{cad_cpt}
The following relations hold:

(1)\, $(\C_{ad})^{'}=\C_{pt}$ and $\mathcal{Z}_2(\C)\subseteq \C_{pt}$.

(2)\, $\mathcal{Z}_2(\C_{pt})=\C_{ad}\vee \mathcal{Z}_2(\C).$
\end{lemma}
\begin{proof}
(1)\,  By \cite[Proposition 3.25]{drinfeld2010braided}, a simple object $X\in \C$ belongs to $(\C_{ad})'$ if and only if it belongs to $\mathcal{Z}_2(\C)^{co}$; that is, if and only if $X\otimes X^{*}\in \mathcal{Z}_2(\C)$. If $X$ is not invertible then $X\otimes X^{*}\cong \1\oplus \delta$ and hence $\delta\otimes X\cong X$, where $\delta$ is unique non-trivial simple object of $\C_{ad}$. Hence $\svect\subseteq \mathcal{Z}_2(\C)$. But by Lemma \ref{cad}, $\C_{ad} \cong \svect$.  This is impossible by \cite[Lemma 5.4]{muger2000galois} which says that if $\svect\subseteq \mathcal{Z}_2(\C)$ then $\delta \otimes Y\ncong Y$ for any $Y\in \C$. Therefore, $(\C_{ad})^{'}\subseteq \C_{pt}$ is pointed. By Proposition \ref{Adjo_Point}, $(\C_{ad})^{'}\supseteq (\C_{pt})^{''}=\C_{pt}\vee \mathcal{Z}_2(\C)$. Hence we have
$$\C_{pt}\supseteq (\C_{ad})^{'}\supseteq\C_{pt}\vee \mathcal{Z}_2(\C)\supseteq \C_{pt},$$
which shows that $(\C_{ad})^{'}=\C_{pt}$ and $\mathcal{Z}_2(\C)\subseteq \C_{pt}$.

(2) By part (1), we have
\begin{equation*}
\mathcal{Z}_2(\C_{pt})=\C_{pt}\cap (\C_{pt})^{'}=\C_{pt}\cap (\C_{ad})^{''}
=\C_{pt}\cap (\C_{ad}\vee \mathcal{Z}_2(\C))=\C_{ad}\vee \mathcal{Z}_2(\C).
\end{equation*}
The last equality follows from \cite[Lemma 5.6]{drinfeld2010braided} and the fact that $\C$ is braided. This proves part (2).
\end{proof}

\section{$N$-Ising categories}\label{nising}

In what follows we shall denote by $\Ii$ the semisimplification of the representation category of $U_{-q}(\mathfrak{sl}_2)$, where $q = \exp(i\pi/4)$.
Then $\Ii$ is an Ising fusion category; see Subsection \ref{sec2.3}.

Recall that there exist exactly 2 non-equivalent such fusion categories, say $\Ii$ and $\Ii^-$. So that $\Ii^-$ is obtained from $\Ii$ by twisting the associator by the 3-cocycle $\alpha$ on $\Zz_2$ such that $\alpha(1, 1, 1) = -1$.

\medbreak
We shall use the notation $\I$ to indicate either of the categories $\Ii$ or $\Ii^-$. As in Subsection \ref{sec2.3} we shall denote by $\delta$ the unique nontrivial invertible object of $\I$ and $Z$ the unique non-invertible simple object.

\medbreak
Let $M \geq 2$ be an even natural number. Consider the fusion subcategory $\C_{M}$ of $\Ii \boxtimes \vect_{\Zz_M}$ generated by the object $Z \boxtimes 1$. The relation \eqref{fr-ising} implies that $\C_M$ has $M/2$ non-invertible simple objects:
\begin{equation}Z_j = Z \boxtimes (2j+1), \quad 0 \leq j \leq \frac{M}{2}-1,
\end{equation}
and $M$ invertible objects:
\begin{equation}\delta^i \boxtimes (2j), \quad 0 \leq i \leq 1, \; 0 \leq j \leq \frac{M}{2}-1.
\end{equation}
Thus $\FPdim Z_j = \sqrt 2$, for all $j=0, \dots, M/2 -1$ and $\FPdim \C_M = 2M$.

\begin{remark}\label{are-braided}
Every fusion category $\C_M$, $M \geq 2$, admits a braiding; to see this it suffices to consider any braiding in $\Ii \boxtimes \vect_{\Zz_M}$ and restrict it to $\C_M$.
\end{remark}

The categories $\C_M$ have generalized Tambara-Yamagami fusion rules. Let us denote by $a = \1 \boxtimes 2 \in \C_M$. Explicitly, the fusion rules of $\C_M$ are determined as follows: the group of invertible objects is a direct product $\langle \delta\rangle \boxtimes \langle a \rangle \cong \Zz_2 \times \Zz_{M/2}$ and
\begin{equation}\label{fr-cm} Z_j \otimes Z_\ell \cong a_{j+\ell + 1} \oplus \delta \, a_{j+\ell + 1}, \quad 0 \leq j, \ell \leq \frac{M}{2}-1.
\end{equation}

\begin{remark}
The categories $\C_M$ are particular cases of the construction in \cite{EM} of fusion categories which are cyclic extensions of fusion categories of adjoint ADE type. Note that the adjoint subcategory of $\C_M$ coincides with the subcategory generated by $\delta$. In particular, $\C_M$ is a $\Zz_M$-extension of the fusion category of adjoint $A_3^{(1)}$ type $\Ii_{ad} \cong \vect_{\Zz_2}$.
\end{remark}

\begin{remark} The construction of the categories $\C_M$ can be generalized replacing the cyclic group $\Zz_M$ by any finite Abelian group $A$ as follows: We may suppose that $A = \Zz_{d_1} \times \dots \times \Zz_{d_r}$, where $d_1,  \dots, d_r \geq 1$. Let $e_1, \dots, e_r$ be the canonical generators of $A$. Then the fusion subcategory of $\Ii \boxtimes A$ generated by the simple objects $Z \boxtimes e_j$, $1\leq j\leq r$, is an $A$-graded extension of $\vect_{\Zz_2}$.
Observe that all the fusion categories arising in this way admit a braiding (c.f. Remark \ref{are-braided}). The examples arising from this construction in fact boil down to the ones obtained from cyclic groups, in view of Theorem \ref{cls-extensions-vec2} below.
\end{remark}

\emph{Let $N \geq 1$. In what follows we shall use the notation $\Ii_N$ to indicate the fusion category $\C_{2^N}$ defined above.}

\begin{example}\label{self-dual} As pointed out before, the category $\Ii_1 = \Ii$ is an Ising fusion category. In particular, it is non-degenerate. The category $\Ii_2$ has two non-isomorphic simple objects $Z_1$ and $Z_2$ of Frobenius-Perron dimension $\sqrt 2$. The group of invertible objects is $\langle \delta \rangle \times \langle a \rangle \cong \Zz_2 \times \Zz_2$ and we have the fusion rules
$$Z_1^* \cong Z_2, \quad Z_1^{\otimes 2} \cong a \oplus \delta a \cong Z_2^{\otimes 2}.$$
In particular, $\Ii_2$ does not contain any Ising fusion subcategory.

More generally, the fusion rules \eqref{fr-cm} imply that $\C_M$ contains a non-invertible self-dual simple object if and only if $M/2$ is odd. If this is the case, such self-dual simple object must generate an Ising fusion subcategory. From the non-degeneracy of Ising fusion categories we obtain, for each $M$ such that $M/2$ is odd, an equivalence fusion categories $\C_M \cong \Ii \boxtimes \B$ or $\C_M \cong \Ii^- \boxtimes \B$, where $\B$ is a pointed fusion category. Furthermore, these are equivalences of braided fusion categories regardless of the choice of the braiding in the category $\C_M$. This feature is generalized in Theorem \ref{fact-cm} below.
\end{example}

\begin{theorem}\label{fact-cm} Let $M \geq 2$ be an even natural number. Suppose that $M = 2^N m$, where $N \geq 1$ and $m \geq 1$ is odd. Then there is an equivalence of fusion categories $\C_M \cong \Ii_N \boxtimes \B$, where $\B$ is a pointed fusion category. Moreover, with respect to any braiding in $\C_M$, this is an equivalence of braided fusion categories for an appropriate braiding in $\Ii_N$. \end{theorem}

\begin{proof} It will be enough to show that $\C_M \cong \Ii_N \boxtimes \B$ as fusion categories. Indeed, if this is the case, then
regardless of the braiding we consider in $\C_M$, the fusion subcategories $\Ii_N$ and $\B$ must centralize each other, since their Frobenius-Perron dimensions are coprime; see \cite[Proposition 3.32]{drinfeld2010braided}.

By assumption,  $\Zz_M$ is the direct sum of the subgroup generated by $m$ and the subgroup $S \cong \Zz_m$ generated by $2^N$. Let $\D_1 \cong \vect_{\Zz_m}$ denote the fusion subcategory of $\C_M$ generated by $\1 \boxtimes S$.

We have an equivalence of fusion categories $\vect_{\Zz_{2^N}} \cong \langle m \rangle \subseteq \vect_{\Zz_M}$, where $\langle m \rangle$ is the fusion subcategory generated by $m$ in $\vect_{\Zz_M}$. Thus the non-invertible simple object $Z \boxtimes m$ of $\C_M$ generates a fusion subcategory $\D_2$ equivalent to $\Ii_N$.

Consider the braiding on $\C_M$ induced by some braiding in $\Ii$ and the trivial half-braiding in $\vect_{\Zz_M}$.
With respect to such braiding, the fusion subcategories $\D_1$ and $\D_2$ centralize each other. In addition, since $\FPdim \D_1 = m$ and $\FPdim \D_2 = 2^{N+1}$ are coprime, then $\D_1 \cap \D_2 \cong \vect$. Therefore, $\D_1 \vee \D_2 \cong \D_1 \boxtimes \D_2$, by \cite[Proposition 7.7]{muegerII}.
Since $\FPdim (\D_1 \boxtimes \D_2) =  2^{N+1}m = \FPdim \C_M$, then  $\C_M = \D_1 \vee \D_2 \cong \D_1 \boxtimes \D_2 \cong \Ii_N \boxtimes \vect_{\Zz_m}$, as was to be shown.
\end{proof}

\medbreak
Let $\omega$ be a 3-cocycle on $\Zz_M$. Recall from Subsection \ref{sec2.2} that $\C_M^\omega$ denotes the fusion category obtained from $\C_M$ by twisting the associator with $\omega$.

\medbreak
It follows from \cite[Lemma 2.12]{EM} that, for every 3 cocycle $\omega$ on $\Zz_M$, the fusion category $\C_M^{\omega}$ has a concrete realization as the fusion subcategory of $\Ii \otimes \vect_{\Zz_M}^\omega$ generated by the simple object $Z \boxtimes 1$.

\medbreak
For every $M$th root of 1, $\zeta \in k^\times$, we shall denote by $\C_{M, \zeta}$ the fusion category obtained from $\C_{M}$ by twisting the associator with the 3-cocycle $\omega_\zeta$ defined by formula \eqref{omega-zm}.
Letting $M = 2^N$, we obtain $2^N$ fusion categories $\Ii_{N, \zeta}$ which are 3-cocycle twists of $\Ii_N = \Ii_{N, 1}$. For $\zeta_1 \neq \zeta_2$, the fusion categories $\Ii_{N, \zeta_1}$ and $\Ii_{N, \zeta_2}$ are non-equivalent as $\Zz_{2^N}$-extensions of $\vect_{\Zz_2}$.   We stress that, for fixed $N$, all the categories $\Ii_{N, \zeta}$ share the same fusion rules.

\begin{definition}
For $N\geq 1$, $\zeta \in \mathbb G_{2^N}$, the category $\Ii_{N, \zeta}$ will be called an \emph{$N$-Ising fusion category}.
\end{definition}






\medbreak
The next theorem shows that the decomposition of $\C_M$ in Theorem \ref{fact-cm} is sharp.

\begin{theorem}\label{nofact-in} Let $N \geq 1$ and let $\zeta \in k^\times$ be a $2^N$th root of 1. Then every proper fusion subcategory of $\Ii_{N, \zeta}$ is pointed. In particular, the category $\Ii_{N, \zeta}$ does not admit any proper exact factorization.
\end{theorem}

\begin{proof} It is enough to show the first statement. Let $\C = \Ii_{N, \zeta}$. Let us identify the universal grading group of $\C$ with the cyclic group $\Zz_{2^N}$ of order $2^N$.
Let $X = Z\boxtimes 1 \in \C_1$, so that $X$ is a faithful simple object of $\C$. Then the rank of $\C_{2m-1}$ is $1$ and the rank of $\C_{2m}$ is $2$, for all $m \geq 1$. Since $2m-1$ is also a generator of $U(\C)$, we have that every non-invertible simple object of $\C$ is faithful. This implies that $\C$ contains no proper non-pointed fusion subcategories, as claimed.
\end{proof}

Recall that a braided fusion category is called \emph{prime} if it contains no non-trivial non-degenerate fusion subcategories.

\medbreak
As a consequence of Theorem \ref{nofact-in} we obtain the primeness of the braided $N$-Ising categories:

\begin{corollary}\label{in-prime} Let $N \geq 1$ and let $\I_N$ be an $N$-Ising fusion category. Assume that $\I_N$ admits a braiding. Then $\I_N$ is prime.
\end{corollary}

\subsection{Braidings on $N$-Ising categories}\label{braided-nisings}

In this subsection we discuss braidings on $N$-Ising fusion categories. If $N = 1$, then $\Ii_{N,\pm 1}$ are Ising fusion categories and therefore they admit (necessarily non-degenerate) braidings.

\begin{remark}
Observe that no braided fusion category equivalent to a 3-cocycle twist of one the categories $\C_M$ can be non-degenerate if $M/2$ is even.  In fact, by  \cite[Lemma 5.4 (ii)]{natale2013faithful}, every non-degenerate fusion category with generalized Tambara-Yamagami fusion rules has a non-invertible self-dual simple object. In particular, with respect to any possible braiding, an $N$-Ising fusion category is non-degenerate if and only if $N = 1$.

\medbreak 	
Let $M\geq 1$ be any even natural number. Consider the braiding in $\C_M$ induced by some fixed braiding in $\Ii$ and the trivial braiding in $\vect_{\Zz_M}$. Then the M\" uger center $\Z_2(\C_M)$ is $\C_M \cap \C_M'$, where $\C_M'$ is the M\" uger centralizer of $\C_M$ in $\Ii \boxtimes \vect_{\Zz_M}$. Since $\C_M$ is generated by  simple object $Z \boxtimes 1$, then $\C_M' = \1 \boxtimes \vect_{\Zz_M}$ and therefore $\Z_2(\C_M) \cong \vect_{\Zz_{M/2}}$ is Tannakian. Hence for this particular braiding, the category $\C_M$ is not slightly degenerate neither.
\end{remark}

Note that, by Lemma \ref{braiding-zm}, each of the categories $\vect_{\Zz_{2^N}}$ and $\vect_{\Zz_{2^N}}^{-1}$ admits a braiding. Hence $\Ii \boxtimes \vect_{\Zz_{2^N}}$ and $\Ii \boxtimes \vect_{\Zz_{2^N}}^{-1}$ admit a braiding and therefore the same holds for their fusion subcategories $\Ii_{N, 1}$ and $\Ii_{N, -1}$.

\begin{remark}\label{pm1pmi} Let $N\geq 1$ and let $\zeta \in \mathbb G_{2^N}$.
Suppose that $\Ii_{N, \zeta}$ admits a braiding. Then $\zeta = \pm 1$ or $\zeta = \pm \sqrt{-1}$.
	
\medbreak
Indeed, the pointed fusion subcategory $(\Ii_{N, \zeta})_{pt}$ is equivalent to $\langle \delta\rangle \boxtimes \langle 2\rangle \cong \vect_{\Zz_2} \boxtimes \vect_{\Zz_{2^{N-1}}}^{\bar\omega}$, where $\bar\omega$ is the 3-cocycle on $\Zz_{2^{N-1}} \cong \langle 2\rangle$ corresponding to the restriction of $\omega_{\zeta}$. Thus $\bar{\omega} = \omega_{\zeta^2}$. Since $\vect_{\Zz_{2^{N-1}}}^{\bar\omega}$ admits a braiding, Lemma \ref{braiding-zm} implies that $\zeta^2 = \pm 1$. Therefore $\zeta = \pm 1$ or $\zeta = \pm \sqrt{-1}$, as claimed.    	

\medbreak In addition, Lemma \ref{cad} implies that the adjoint subcategory $(\Ii_{N,\zeta})_{ad}$ is equivalent to $\svect$ as braided fusion categories.
\end{remark}

\begin{lemma}\label{2ising-sd} Let $\zeta \in \mathbb{G}_{4}$. Then a 2-Ising fusion category $\Ii_{2, \zeta}$ admits a braiding if and only if $\zeta = \pm 1$.
\end{lemma}

\begin{proof}
As observed in Remark \ref{pm1pmi}, both $\Ii_{2, 1}$ and $\Ii_{2, -1}$ admit a braiding.
	
Suppose conversely that $\Ii_{2, \zeta}$ admits a braiding. As pointed out in Remark \ref{pm1pmi}, $\zeta = \pm 1$ or $\zeta = \pm \sqrt{-1}$. If $\zeta = \pm \sqrt{-1}$, then the pointed subcategory $\langle 2\rangle$	must be equivalent as a fusion category to $\vect_{\Zz_2}^{-1}$. In particular, $\langle 2\rangle$ is non-degenerate, which contradicts the primeness of $\Ii_{2, \zeta}$ (see Corollary \ref{in-prime}). Then we get that $\zeta = \pm 1$.
\end{proof}

\begin{lemma}\label{sd-nising} Suppose that $\I_N$, $N \geq 1$, is a braided $N$-Ising fusion category such that its M\" uger center contains a fusion subcategory braided equivalent to the category $\svect$ of super-vector spaces. Then $\I_{N}$ is slightly degenerate.
\end{lemma}

\begin{proof} Let $\C = \I_{N}$. Then the M\" uger center $\Z_2(\C)$ is a pointed fusion category. Since the group of invertible objects of $\C$ coincides with $\langle \delta\rangle \boxtimes \langle 2\rangle \cong \Zz_2 \times \Zz_{2^{N-1}}$ and  $\Z_2(\C) \cap \C_{ad} \cong \vect$, then the group of invertible objects of $\Z_2(\C)$ is cyclic. Combined with Lemma \ref{pointed-svec} below, the assumption implies that $\Z_2(\C) \cong \svect$ as braided fusion categories. Thus $\C$ is slightly degenerate.
\end{proof}

The next example shows that, for all $N > 2$, the categories $\Ii_{N, -1}$ admit slightly degenerate braidings.

\begin{example}\label{sd-induced} Suppose that $N > 2$. Recall from Example \ref{braided-z2n} that the fusion category $\vect_{\Zz_{2^N}}^{\zeta}$ admits a non-degenerate braiding if and only if $\zeta = -1$.
	
Consider the braiding in $\Ii \boxtimes \vect_{\Zz_{2^N}}^{-1}$ induced by any fixed braiding in $\Ii$ and a non-degenerate braiding in $\vect_{\Zz_{2^N}}^{-1}$. Then $\Ii \boxtimes \vect_{\Zz_{2^N}}^{-1}$ is non-degenerate.

Regard $\C = \Ii_{N, -1}$ as a braided fusion category with the braiding induced from $\Ii \boxtimes \vect_{\Zz_{2^N}}^{-1}$. Hence $\Z_2(\C) = \C \cap \C'$. Moreover, since $\FPdim \Ii_{N, -1} = 2^{N+1}$ and $\Ii \boxtimes \vect_{\Zz_{2^N}}^{-1}$ is non-degenerate, then $\FPdim \C' = 2$.
Since $\C$ is degenerate, then $\C' \subseteq \C$.

\medbreak
Since $\Ii$ is non-degenerate, then the non-trivial simple object of $\C'$ must be of the form $Y \boxtimes a$, where $a\in \Zz_{2^N}$ is the unique element of order $2$ and  $Y = \1$ or $Y = \delta$. Suppose that $Y = \1$. Then $\1 \boxtimes a$ centralizes $Z\boxtimes 1$ and therefore $a$ centralizes $1 \in \Zz_{2^N}$. This implies that $a$ centralizes $\vect_{\Zz_{2^N}}^{-1}$, which contradicts the non-degeneracy of $\vect_{\Zz_{2^N}}^{-1}$.
Thefore $Y = \delta$.

\medbreak
Let $q$ be the quadratic form on $\langle \delta \rangle \boxtimes \Zz_{2^{N-1}}$ associated to the induced braiding in $\C_{pt}$.
The observations in Example \ref{braided-z2n}, imply that $q(a) = 1$. Since $\delta \boxtimes 0$ is the only non-trivial object of $\C_{ad} \cong \svect$, then $q(\delta \boxtimes 0) = -1$. Using that $\delta \boxtimes 0$ centralizes $\C_{pt}$, we get that $q(\delta \boxtimes a) = q(\delta \boxtimes 0) q(\1\boxtimes a) = -1$. This implies that $\Z_2(\C) \cong \svect$. Then $\C = \Ii_{N, -1}$ is slightly degenerate.

\medbreak
If $N = 2$ then $a = 2$ and, as observed in Example \ref{braided-z2n}, $\langle a \rangle \cong \svect$. Hence $\Z_2(\Ii_{2, -1}) = \langle \delta \boxtimes a\rangle \cong \Rep \Zz_2$ is a Tannakian subcategory.

\medbreak
Observe that in these examples the pointed subcategory of $\Ii_{N, -1}$ is  $\langle \delta \rangle \boxtimes \langle 2 \rangle \cong \svect \boxtimes \vect_{\Zz_{2^{N-1}}}$.
\end{example}

\begin{lemma} Let $N > 2$. Consider a braiding in $\Ii_{N, \zeta}$ induced from a braiding in $\Ii\boxtimes \vect_{\Zz_{2^N}}^\zeta$. Then $\Ii_{N, \zeta}$ is slightly degenerate if and only if the induced braiding in $\vect_{\Zz_{2^N}}^\zeta$ is non-degenerate. If this is the case, then $\zeta = -1$.
\end{lemma}

\begin{proof} By Lemma \ref{braiding-zm}, $\zeta = \pm 1$. In view of Example \ref{braided-z2n}, it will be enough to prove the first statement. The 'if' direction was shown in Example \ref{sd-induced}. Suppose conversely that $\Ii_{N, \zeta}$ is slightly degenerate.
Note that with respect to any braiding in $\Ii \boxtimes \vect_{\Zz_{2^N}}^\zeta$, the subcategory $\1 \boxtimes \vect_{\Zz_{2^N}}^\zeta$ must centralize $\Ii \boxtimes 0$ projectively. In view of \cite[Proposition 3.32]{drinfeld2010braided}, this implies that
if $a = 2^{N-1}$ is the unique element of order 2 of $\Zz_{2^N}$, then $\1\boxtimes a$ centralizes $Z \boxtimes 0$.

If $\1\boxtimes \vect_{\Zz_{2^N}}^\zeta$ is degenerate, then its M\" uger center must contain $\1\boxtimes a$ and therefore $\1 \boxtimes a$ centralizes $Z\boxtimes 1$. Since $\1\boxtimes a \in \Ii_{N, \zeta} = \langle Z \boxtimes 1\rangle$, then $\1 \boxtimes a \in \Z_2(\Ii_{N, \zeta})$. Hence $\Z_2(\Ii_{N, \zeta}) = \langle \1 \boxtimes a\rangle$. But, from Formula \eqref{sq-bd},  $q(a) = 1$, where $q$ is the quadratic form in $\Zz_{2^N}$ corresponding to the induced braiding in $\1 \boxtimes \vect_{\Zz_{2^N}}^\zeta$. Then $\Z_2(\Ii_{N, \zeta})$ is Tannakian against the assumption.

This shows that  $\vect_{\Zz_{2^N}}^\zeta$ must be  non-degenerate and  finishes the proof of the lemma.
\end{proof}

\begin{remark}	Suppose $\C$ is a slightly degenerate $N$-Ising fusion category and $N > 2$.
We have $\C_{pt} \cong \C_{ad} \boxtimes \D$, where $\D = \langle \1 \boxtimes 2 \rangle$ is a pointed fusion category whose group of invertible objects is cyclic of order $2^{N-1}$.  This is in fact an equivalence of braided fusion categories since, by Lemma \ref{cad_cpt}, $\C_{ad}$ centralizes $\C_{pt}$. Therefore \begin{equation}\label{equiv1}\Z_2(\C_{pt}) \cong \C_{ad} \boxtimes \Z_2(\D).\end{equation}
	
On the other hand, using again Lemma \ref{cad_cpt} and \cite[Proposition 7.7]{muegerII}, we find \begin{equation}\label{equiv2}\Z_2(\C_{pt}) = \C_{ad} \vee \Z_2(\C) \cong \C_{ad} \boxtimes \Z_2(\C).\end{equation}	
	
From \eqref{equiv1} and \eqref{equiv2} we obtain that $\FPdim \Z_2(\D) = 2$. Furthermore, if $\Z_2(\D) \cong \svect$, then Lemma \ref{pointed-svec} implies that $\svect$ is a direct factor of $\D$. This is possible only if $N = 2$.
	
\medbreak Since $N > 2$, then $\Z_2(\langle \1 \boxtimes 2\rangle)$ is Tannakian of dimension $2$. Hence $\Z_2(\langle \1 \boxtimes 2\rangle) \cong \langle \1\boxtimes 2^{N-2}\rangle \cong \Rep \Zz_2$ and the non-trivial object of $\Z_2(\C)$ is $\delta \boxtimes 2^{N-2}$.
\end{remark}

\section{The structure of braided extensions of $\vect_{\Zz_2}$}\label{structure}

Suppose that $\B$ is a pointed braided fusion category. Corollary A. 19 of \cite{drinfeld2010braided} states that if the M\" uger center $\Z_2(\B)$ of $\B$ coincides with the category $\svect$ of super-vector spaces, then the M\" uger center is a direct factor of $\B$, that is, $\B \cong \svect \boxtimes \B_0$, for some pointed (necessarily non-degenerate in this case) braided fusion category $\B_0$. However, the proof of \cite[Corollary A. 19]{drinfeld2010braided} only uses the fact that $\svect \subseteq \Z_2(\B)$, in other words, it actually proves the following:

\begin{lemma}\label{pointed-svec} Let $\B$ be a pointed braided fusion category. Suppose that the M\" uger center of $\B$ contains a fusion subcategory $\D$ braided equivalent to the category $\svect$ of super-vector spaces. Then $\B \cong \D \boxtimes \B_0$, for some pointed braided fusion category $\B_0$.
\end{lemma}

Let $\vect_{\Zz_{2M}}^{\alpha}$ be the pointed fusion category  with associativity constraint  given by the 3-cocycle $\alpha$, where $$\alpha(a, b, c) =\begin{cases}1, \; \; \qquad \qquad b+c < 2M, \\ \textrm{exp}(\frac{2i\pi a}{M}), \quad  b+c \geq 2M.\end{cases}$$

Consider the fusion category $\D_{2M}$ of $\Ii \boxtimes \vect_{\Zz_{2M}}^{\alpha}$ generated by the simple object $Z \boxtimes 1$. Let $(\D_{2M})_\E$ be the de-equivariantization of the fusion category $\D_{2M}$ by its (central) subcategory $\E$ generated by the invertible object $\delta \boxtimes M$.

\medbreak
The following result is a special instance of the classification of cyclic extensions of fusion categories of adjoint ADE type in \cite{EM}.

\begin{theorem}\label{cls-cyclic-ext}\emph{(\cite[Lemma 3.10]{EM}.)}
Up to twisting the associator by a 3-cocycle $\omega$ on $\Zz_M$, every $\Zz_M$-extension of $\vect_{\Zz_2}$, $\otimes$-generated by a simple object of Frobenius-Perron dimension less than 2, is equivalent as a fusion category to some of the categories $\C_{M}$ or, if $4$ divides $M$, to some of the categories $(\D_{2M})_\E$.
\end{theorem}

As an application of Theorem \ref{cls-cyclic-ext}, we obtain:

\begin{theorem}\label{braided-cyclic} Let $\C$ be a non-pointed braided fusion category and suppose that $\C$ is a $\Zz_M$-extension of the fusion category $\vect_{\Zz_2}$. Then $\C$ is equivalent as a fusion category to $\C_{M}^\omega$, for some 3-cocycle $\omega$ on $\Zz_M$.
\end{theorem}

\begin{proof} By assumption the braided fusion category $\C$ is nilpotent. Since $\C$ is not pointed, then $\C_{ad} \cong \vect_{\Zz_2}$ and therefore $U(\C) \cong \Zz_M$. Then \cite[Theorem 4.7]{natale2013faithful} implies that $\C$ has a faithful simple object $X$ and in addition $X$ is not invertible. Since the homogeneous components of the $\Zz_M$-grading of $\C$ have dimension $2$, then $\FPdim X = \sqrt 2$ (see Theorem \ref{General_TY_Cat}). Hence $\C$ is $\otimes$-generated by a simple object of Frobenius-Perron dimension less than 2.

\medbreak
In view of  Theorem \ref{cls-cyclic-ext} we may assume that $\C$ is equivalent to a 3-cocycle twist of one of the fusion categories $(\D_{2M})_\E$, where $M$ is divisible by $4$.

Note that the equivariantization functor $F: \D_{2M} \to (\D_{2M})_\E$ takes a simple object of Frobenius-Perron dimension $\sqrt{2}$ of $\D_{2M}$ to a simple object (of the same Frobenius-Perron dimension) of $(\D_{2M})_\E$. Then $F$  induces a surjective group homomorphism $G(\D_{2M}) \to G((\D_{2M})_\E)$ whose kernel is the subgroup $\langle\delta \boxtimes M \rangle$ generated by $\delta \boxtimes M$. Hence we obtain a group isomorphism $G((\D_{2M})_\E) \cong G(\D_{2M})/\langle\delta \boxtimes M \rangle$. But $G(\D_{2M}) = \langle \delta \rangle \boxtimes \langle 2 \rangle$, so that $G((\D_{2M})_\E) \cong \Zz_M$ is cyclic of order $M$.

Then the group of invertible objects of $\C$ is cyclic of order $M$.
Since $\C$ is not pointed, then $\C$ has generalized Tambara-Yamagami fusion rules.
Then the group of invertible objects of $\C$, being cyclic, must contain a unique subgroup of order $2$. This subgroup is necessarily the group of invertible objects of the adjoint subcategory $\C_{ad} \cong \vect_{\Zz_2}$.

By Lemmas \ref{cad} and \ref{cad_cpt},
$\C_{ad} \cong \svect$ as braided fusion categories and $\Z_2(\C_{pt}) = \C_{ad} \vee \Z_2(\C)$. Then, by Lemma \ref{pointed-svec}, $\C_{pt} \cong \C_{ad} \boxtimes \B$, for some pointed fusion category $\B$. Since $G(\C)$ is cyclic, we obtain that  $\B$ has odd dimension $n$. This implies that $M/2 = n$ is odd, agains the assumption.
The proof of the theorem is now complete.  \end{proof}

\begin{remark} The proof of Theorem \ref{braided-cyclic} shows that (twistings of) the fusion categories $(\D_{2M})_\E$ are not braided unless $M/2$ is odd, in which case they are equivalent to a twisting of the fusion category $\C_M$.
When $M = 4$, $(\D_{2M})_\E$ has Fermionic Moore-Reed fusion rules. It is known that there are four fusion categories admitting these fusion rules and none of them is braided; see \cite{Bonderson2007}, \cite{liptrap2010generalized}.
\end{remark}

\medbreak
The following is the main result of this section:

\begin{theorem}\label{cls-extensions-vec2} Let $\C$ be a non-pointed braided fusion category and suppose that $\C$ is an extension of a rank 2 pointed fusion category. Then $\C$ is equivalent as a fusion category to $\I_{N} \boxtimes \B$, for some $N\geq 1$, where $\I_N$ is a braided $N$-Ising fusion category, and $\B$ is a pointed braided fusion category.
Furthermore, the categories $\I_{N}$ and $\B$ projectively centralize each other in $\C$.
\end{theorem}

\begin{proof} Let $U(\C)$ be the universal grading group of $\C$, denoted additively. Then $U(\C)$ is an Abelian group and $\C = \bigoplus_{a \in U(\C)} \C_a$, with $\C_0 = \C_{ad} \cong \vect_{\Zz_2}$. Then $\C_{ad} \cong \svect$ as braided fusion categories. We shall denote by $\delta$ the unique non-invertible simple object of $\C_{ad}$.

Let us identify $U(\C)$ with a direct sum of cyclic groups $\Zz_{d_1} \oplus \dots \oplus \Zz_{d_r}$, where the integers $2\leq d_1, \dots, d_r$ are such that
$d_j\vert d_{j+1}$, for all $j = 1, \dots, r-1$.
Let $e_i\in U(\C)$, $1\leq i \leq r$, be the canonical generators: $e_i$ has 1 in the $i$th component and 0 in the remaining components.

For each $1\leq i \leq r$, let $\C_{e_i}$ be the homogeneous component of degree $e_i$ of $\C$.
Write the set $\{1, \dots, r\}$ as a disjoint union $\{i_1, \dots, i_p\} \cup \{j_1, \dots, j_q\}$, where $p+q = r$ and the indices $i_1, \dots, i_p, j_1, \dots, j_q$ are such that
\begin{equation}\label{desigualdades} i_1\leq \dots \leq i_p, \quad  j_1\leq \dots \leq j_q, \end{equation}
and the homogeneous components $\C_{e_{i_\ell}}$, $1\leq \ell \leq p$, contain a non-invertible simple object $Z_{i_\ell}$, and the components $\C_{e_{j_s}}$, $1\leq s \leq q$, contain two non-isomorphic invertible objects $a_{j_s}$ and $b_{j_s}$.

\begin{claim}\label{claim} The $p + 2q$ simple objects
\begin{equation}\label{generators} Z_{i_1}, \dots, Z_{i_p}, a_{j_1}, b_{j_1}, \dots, a_{j_q}, b_{j_q},\end{equation} generate the fusion category $\C$.
\end{claim}

\begin{proof}[Proof of the claim] Let $X$ be a simple object of $\C$ and suppose that $X \in \C_a$, $a \in U(\C)$. Since $e_1, \dots, e_r$ generate $U(\C)$, then $a = {t_1}e_1 + \dots +{t_r} e_r$, for some non-negative integers $t_1, \dots, t_r$. Then the tensor product
\begin{equation}\label{tensor-product}Z_{i_1}^{\otimes t_{i_1}}\otimes \dots \otimes Z_{i_p}^{\otimes t_{i_p}} \otimes x_{j_1}^{t_{j_1}} \dots x_{j_q}^{t_{j_q}}\end{equation}
belongs to $\C_a$, where, for all $1\leq s \leq q$, $x_{j_s} = a_{j_s}$ or $b_{j_s}$.

\medbreak
If $X$ is the unique simple object of $\C_a$ up to isomorphism, then the tensor product \eqref{tensor-product} must be isomorphic to a direct sum of copies of $X$ and in particular $X$ is a simple constituent of \eqref{tensor-product} and therefore it belongs to the fusion subcategory generated by \eqref{generators}.

Note in addition that such a non-invertible simple object $X$ of $\C$ must exist, because $\C$ is not pointed. Thus if $t_1, \dots, t_r$ are chosen as above, then \eqref{tensor-product} does not contain any invertible  constituent. Hence some of the tensor factors in \eqref{generators} must be non-invertible, that is, $p \geq 1$. Since $Z_{i_1} \otimes Z_{i_1}^* \cong \1 \oplus \delta$, then we find that $\delta$ belongs to the fusion subcategory generated by \eqref{generators}.

\medbreak
Suppose next that the simple object $X \in \C_a$ is invertible. Then the only simple objects of $\C_a$ are, up to isomorphism, $X$ and $\delta X$. Also in this case, at least one of the objects $X$ or $\delta X$ is a simple constituent of \eqref{tensor-product} and therefore it belongs to the fusion subcategory generated by \eqref{generators}. Then so does the other, because $\delta$ belongs to this subcategory. This proves the claim.
\end{proof}

By Corollary \ref{transitive-action}~(1), the action of the group of invertible objects of $\C$ on the isomorphism classes of non-invertible simple objects is transitive. Then, for all $1\leq \ell \leq p$, $$Z_{i_1} \otimes Z_{i_\ell}^* \cong g_{\ell} \oplus \delta g_{\ell},$$ where $\1 \neq g_\ell$ is an invertible object such that \begin{equation}\label{transitive}g_\ell \otimes Z_{i_1} \cong Z_{i_\ell}.\end{equation} In particular $g_1 = \delta$. Then $g_\ell$ and $\delta g_\ell$ are, up to isomorphism, the unique simple objects of $\C_{e_{i_1} - e_{i_\ell}}$.

\medbreak
Let $\tilde\B$ be the pointed fusion subcategory of $\C$ generated by the invertible objects
\begin{equation}\label{gen-b}a_{j_1}, b_{j_1}, \dots, a_{j_q}, b_{j_q}, g_1, g_2, \dots, g_p. \end{equation}

Since $\delta = g_1$ generates $\C_{ad}$, then $\svect \cong \C_{ad} \subseteq \tilde\B$. But by Lemma \ref{cad_cpt},
$\C_{ad}$ centralizes $\tilde\B$. Lemma \ref{pointed-svec} implies that $\tilde\B \cong \C_{ad} \boxtimes \B_0$ for some pointed fusion category $\B_0$. Note that the degree of homogeneity $b$ of a simple object of $\B_0$ is of the form
\begin{align*}b & = s_2(e_{i_1} - e_{i_2}) + \dots + s_p(e_{i_1} - e_{i_p})+ n_1e_{j_1} + \dots + n_qe_{j_q} \\ & = he_{i_1}
- s_2e_{i_2} - \dots - s_p e_{i_p} + n_1e_{j_1} + \dots + n_qe_{j_q},\end{align*} for some non-negative integers $s_2, \dots, s_p, n_1,\dots, n_q$, where $h = s_2 + \dots + s_p$.

\medbreak
Let $Z = Z_{i_1}$. Relation \eqref{transitive} and Claim \ref{claim} imply that the fusion subcategory $\langle Z\rangle$ generated by $Z$ and $\B_0$ generate $\C$. By commutativity of the fusion rules of $\C$, we obtain that every simple object $Y$ of $\C$ decomposes in the form \begin{equation}\label{factoriz-simple}Y \cong X \otimes g,\end{equation} for some simple object $X$ of $\langle Z \rangle$ and some invertible object $g \in \B_0$.

\medbreak
Suppose that $X, X' \in \langle Z \rangle$ and $g, g' \in \B_0$ are simple objects such that
\begin{equation}\label{iso} X\otimes g \cong X'\otimes g'.
\end{equation}

Then $X \otimes g(g')^{-1} \in \langle Z\rangle$ and thus $g(g')^{-1}$ is a simple constituent of $Z^{\otimes m}$, for some $m \geq 0$. In particular, $g(g')^{-1}$ is homogeneous of degree ${me_{i_1}}$.

On the other hand, $g(g')^{-1} \in \B_0$. Then
$$me_{i_1} = he_{i_1}
- s_2e_{i_2} - \dots - s_p e_{i_p} + n_1e_{j_1} + \dots + n_qe_{i_q},$$
for some non-negative integers $s_2, \dots, s_p, n_1,\dots, n_q$, and $h = s_2 + \dots + s_p$.
Therefore $d_{i_2}\vert s_2, \dots, d_{i_p}\vert s_p$ and $d_{j_1}\vert n_1, \dots, d_{j_q}\vert n_q$.
From condition \eqref{desigualdades}, we have that $d_{i_1}\vert d_{i_2}\vert \dots \vert d_{i_p}$. Hence $d_{i_1}\vert h$ and $g(g')^{-1} \in \C_0 = \C_{ad}$. Therefore $g(g')^{-1} \cong  \1$, by the definition of $\B_0$. Then $g \cong g'$ and also $X \cong X'$, by \eqref{iso}.

\medbreak We have thus shown that the factorization \eqref{factoriz-simple} of a simple object of $\C$ is unique up to isomorphism. By \cite[Theorem 3.8]{gelaki2017exact}, $\C$ is an exact factorization $\C = \langle Z\rangle \bullet \B_0$.  Since $\C$ is braided, then $\C \cong \langle Z \rangle \boxtimes \B_0$ as fusion categories and the categories $\langle Z\rangle$ and $\B_0$ projectively centralize each other, by \cite[Corollary 3.9]{gelaki2017exact}.
Since $\langle Z \rangle$ is a cyclic extension of $\vect_{\Zz_2}$, then Theorems \ref{braided-cyclic} and \ref{fact-cm} imply that $\langle Z \rangle \cong \Ii_{N, \zeta} \boxtimes \D$, for some $N \geq 1$, where $\zeta$ is a $2^N$th root of 1, and $\D$ is a pointed braided fusion category, such that $\Ii_{N, \zeta}$ and $\D$ centralize each other. Letting $\B = \D \boxtimes \B_0$, we obtain the theorem.
\end{proof}

\section*{Acknowledgements}
J. Dong is partially supported by the startup foundation for introducing talent of NUIST (Grant No. 2018R039) and the Natural Science Foundation of China (Grant No. 11201231). S. Natale is partially supported by CONICET and Secyt-UNC. The work of S. Natale was done in part during visits to NUIST in Nanjing, and ECNU in Shanghai; she thanks both mathematics departments for the outstanding hospitality.

\end{document}